\documentclass[12pt,a4paper]{amsart}

\usepackage{amsmath,amssymb}
\usepackage{mathpazo}
\usepackage{amsthm}
\usepackage{comment}

\newcommand{\re}{\mathrm{Re}}
\newtheorem{thm}{Theorem}[section]

\newcommand{\R}{\mathbb{R}}
\newcommand{\C}{\mathbb{C}}
\newcommand{\N}{\mathbb{N}}
\newcommand{\Z}{\mathbb{Z}}
\newcommand{\ceil}[1]{\lceil #1 \rceil}

\address{Department of Mathematics, Graduate School of Science, 
Kyoto University, Kyoto 606-8502, Japan}
\email{kato.ryo.78w@st.kyoto-u.ac.jp}

\begin{document}
\title{A remark on the Wiener-Ikehara Tauberian Theorem}
\author{Ryo Kato}
\date{}

\maketitle

\textbf{Abstract.} In this paper we point out that the proof of 
Kable's extension of the Wiener-Ikehara Tauberian theorem can 
be applied to the case where the Dirichlet series has a pole 
of order "$l/m$" without much modification
(Kable proved the case $l = 1$).

\section{Introduction}
\label{sec:introduction}

We use the notation $\N,\Z,\R,\C$ for the sets of positive integers, 
all integers, real numbers, complex numbers respectively.  
Let  $\R_{> 0}$ denote the set of positive real numbers. 
For $x \in \R$ we put 
$\ceil{x} = \min\{ n \in \mathbb{Z} \ | \ n \ge x \}$.

Let $\{a_{n}\}$ be a sequence of non-negative real numbers, 
$d$ a positive real number, and $m, l \in \mathbb{N}$. Suppose that the Dirichlet series
\begin{equation*}
L(s) = \sum_{n=1}^{\infty} \frac{a_{n}}{n^s}   
\end{equation*}
converges absolutely for $\re(s) > d$. Also suppose that $L^m$ 
has a meromorphic continuation to an open set containing the 
closed half-plane $\re(s) \ge d$ and holomorphic except for 
a pole of order $l$ at $s = d$. In this paper, if $L^m$ has a 
pole of order $l$, then we say that $L$ has a pole of order $l/m$. 
For functions $f$ and $g : \R \to \R $, we denote 
$f(x) \sim g(x)$ if $\lim_{x \to +\infty} {f(x)}/{g(x)} = 1$.

Our purpose is to determine the asymptotic behavior of 
$\sum_{n \le X} a_{n}$ as $X \to \infty$ by properties of $L(s)$.

In the case where $L$ has a simple pole at $s = d$ (i.e., $m = l = 1$),  
Wiener and Ikehara proved that
\begin{equation*}
 \sum_{n \le X} a_{n} \sim \frac{A}{d}X^d
\end{equation*}
as $X \to \infty$, where $A$ is residue of $L$ at $s = d$. 
This result was proved in 1932(\cite{wiener}) and is called the Wiener-Ikehara theorem.

In the general case, 
an extension was given by Delange (\cite[p.235, TH\'EOR\`EM III]{delange}) 
in 1954. Delange considered the case where the order of the
pole is a positive real number in some sense. However,  
to apply his theorem to $L$ satisfying above conditions, 
an extra condition about zeros of $L^m$ on $\re(s) = d$ is required.
Kable has given an extension for the case where the order of the pole is
$1/m$  without the condition. 
In \cite{kable}, he used the notion of functions of bounded variation.

The result of Delange and Kable is as follows.
Let  $\alpha$ = $1/m$ or $l$.
If $\alpha = 1/m$, then suppose that the residue of $L^m$ at $s = d$ is $A^m$, where $A > 0$. 
If $\alpha = l$, then let $A = \lim_{s \to d} L(s)(s -d)^l \ (A>0)$.
Then
\begin{equation}
\sum_{n \le X} a_{n} \sim \frac{AX^d}{d\Gamma(\alpha)(\log(X))^{1-\alpha}} \label{1}
\end{equation}
as $X \to \infty$.
We shall point out in this paper that the proof of Kable's result 
works for the case, without much modification, 
where the order of the pole is $l/m$, and 
that (\ref{1}) holds with $\alpha=l/m$.

The organization of this paper is as follows. In Section \ref{sec:preliminaries}, 
we define symbols and functions used in this paper. Then we slightly
extend \cite[p.140, T{\footnotesize HEOREM} 1]{kable}.
In Section \ref{sec:main-theorem}, we apply the result in Section \ref{sec:preliminaries}  
to obtain the same result as \eqref{1} for the case $\alpha = l/m$. 
In Section \ref{sec:an-example}, we give an example 
of a Dirichlet series which has a pole of order $2/3$ and apply 
the main theorem (Theorem \ref{thm:main-theorem}).

\section{Preliminaries}
\label{sec:preliminaries}

Let $\lambda$ be $(2 \pi)^{-\frac{1}{2}}$ times the Lebesgue measure 
on $\R$ and $\mathcal{S}(\R)$ the space of Schwartz functions on $\R$. 
For $\Phi \in  \mathcal{S}(\R)$ define the Fourier transform by 
\begin{equation*}
\mathcal{F}(\Phi)(t) = \int_{\R} \Phi(u) e^{-iut} d\lambda(u). 
\end{equation*}
We define the inverse Fourier transform by 
\begin{equation*}
\mathcal{F}^{-1}(\Phi)(t) = \int_{\R} \Phi(u) e^{iut} d\lambda(u). 
\end{equation*}
Then 
\begin{equation*}
\mathcal{F}^{-1}(\mathcal{F}(\Phi))(t) = \Phi(t) \quad i.e.,\ \ 
\Phi(t) = \int_{\R} \mathcal{F}(\Phi)(u) e^{iut} d\lambda(u). 
\end{equation*}
For $\Phi, \Psi \in \mathcal{S}(\R)$, we define the convolution by 
\begin{equation*}
(\Phi * \Psi)(t) = \int_{\R} \Phi(t - u) \Psi(u) d \lambda (u). 
\end{equation*}
It is well known (see \cite[p.183, 7.2 Theorem]{fourier}) that
\begin{equation*}
\mathcal{F}(\Phi * \Psi) = \mathcal{F}(\Phi) \mathcal{F}(\Psi). 
\end{equation*}
%

%p142Prop1.11 

For $s \in \C -(-\infty , 0]$ we choose the branch of $\log(s)$ 
so that $-\pi < arg(s) < \pi$. Let $\alpha \in \R$ and $x \in \C$. 
For $s \in \C - \{ s \in \C \ | \ s - x \in (\infty, 0] \}$ 
we define $(s-x)^{\alpha} = \exp(\alpha \log (s-x))$.
Then $(s-x)^\alpha$ is holomorphic on 
$\C - \{ s \in \C \ | \ s - x \in (\infty, 0] \}$ 
and has positive real values for $s \in \{s \in \C \ | \ s -x \in \R_{> 0}\}.$
 
Let $\mathcal{P}([a, b])$ denote the set of all partitions of $[a, b]$
(i.e., sequences $x_0=a<x_1<\cdots<x_n=b$).
For a function  $f :[a, b] \to \R$ we define 
$V_{a}^{b} f \in \R \cup \{\infty\}$ by 
\begin{equation*}
V_{a}^{b} f = \sup \left\{ \sum_{i = 1}^{n} 
\left| f(x_{i}) - f(x_{i-1}) \right| \bigg| 
P = \{x_{i} \ | \ i = 0,1,\cdots,n\} \\
\in \mathcal{P}([a, b])
\right\}. 
\end{equation*}
We say that a function $f :[a, b] \to \R$ is of bounded variation 
if $V_{a}^{b} f < + \infty$.
Also we say that a function $f :[a, b] \to \C$ is of bounded variation 
if the real part and the imaginary part of $f$ are of bounded variation.
If a function $f :\R \to \mathbb{C}$  is of bounded 
 variation on any closed-interval $[a , b]$, then
we say that $f$ is locally of bounded variation.

The following theorem is an extension of the Wiener-Ikehara 
theorem and plays a key role to prove the main theorem in Section 3.
This theorem slightly extends \cite[p.140, T{\footnotesize HEOREM} 1]{kable}.
 
\begin{thm}
\label{thm:preliminary-theorem}
Let $f : \mathbb{R} \to [0,\infty)$ be a non-decreasing function such that
the Laplace transform 
\begin{equation*}
F(s) = \int_{0}^{\infty} f(u) e^{-su} du 
\end{equation*}
converges absolutely for $\re(s) > 1$. 
Let $\alpha, \alpha_{i} \in \mathbb{R}_{ > 0} ( i = 1, 2, \cdots, r ), 
A \in \mathbb{R}_{> 0}, A_{i} \in \mathbb{C} $ and  
$\alpha > \alpha_{i}$ for all $i$. We define  
\begin{equation*}
G(s) = F(s) - \frac{A}{(s - 1)^{\alpha}} - \sum_{i=1}^{r} \frac{A_{i}}{(s - 1)^{\alpha_{i}}} 
\end{equation*}
for $\re(s) > 1$.
Suppose that the function G extends continuously to the set 
$\{s \in \mathbb{C}$ $|$ $\re(s) \ge 1\}$. 
If $\alpha < 1$ then assume, in addition, 
that once so extended, the function $ t \mapsto G(1 + it)$ is locally of bounded variation.
Then
\begin{equation*}
\lim_{u \to \infty} u^{1 - \alpha} e^{-u} f(u) = \frac{A}{\varGamma(\alpha)}. 
\end{equation*}
\end{thm}
\begin{proof}

Define a function $h : \mathbb{R} \to [ 0,\infty)$ 
by $h(u) = u^{1 - \alpha} e^{-u} f(u)$ and let 
$C_{i} = A_{i} / \varGamma(\alpha_{i}), \ C = A / \varGamma(\alpha)$.
By \cite[p.140, L{\footnotesize EMMA} 2]{kable} 
\begin{equation*}
\frac{1}{(s - 1)^{\alpha}} = \frac{1}{\varGamma(\alpha)} 
\int_{0}^{\infty} e^{-(s-1) u} u^{\alpha - 1} du
\end{equation*}
for $\re(s) > 1$.  By definition
\begin{equation*}
F(s) = \int_{0}^{\infty} h(u) e^{-(s-1) u} u^{\alpha - 1} du.
\end{equation*}
It follows that 
\begin{equation*}
G(s) = \int_{0}^{\infty} \left(h(u) - C - \sum_{i=1}^{r} 
\frac{C_{i}}{u^{\alpha - \alpha_{i}}} \right) e^{-(s-1)u} u^{\alpha - 1} du
\end{equation*}
for $\re(s) > 1$.
Let $\Psi$ be an even Schwartz function with 
compact support and $\Phi = \mathcal{F}(\Psi)^2$
.
Since $\mathcal{F}(\Phi) = \Psi*\Psi$, 
$\mathcal{F}(\Phi)$ is even and has compact support.

Suppose that $\sigma > 1$ and $x > 0$. Since 
\begin{equation*}
F(\sigma) \ge \int_{x}^{\infty} f(u) e^{-\sigma u} du  \ \ , \ \ f(u) \ge 0 
\end{equation*}
and $f(u)$ is monotone increasing, we obtain $f(x) \leq \sigma e^{\sigma x}F(\sigma)$. Hence,
\begin{equation}
h(x) \leq \sigma F(\sigma) e^{(\sigma - 1) x} x^{1 - \alpha} \label{h,h}     
\end{equation}
for  any $x > 0$.

Let $\varepsilon > 0$ be a small number.
We choose $\sigma$ so that $0 < \sigma - 1 < \varepsilon$. 
For any $v \in \mathbb{R}_{ > 0}$, by using \eqref{h,h}, we obtain
\begin{equation}
\int_{0}^{\infty} h(u) \Phi(v -u) e^{-\varepsilon u} u^{\alpha - 1} du 
\leq \sigma F(\sigma) \int_{0}^{\infty} \Phi(v -u) e^{-\big( \varepsilon - (\sigma - 1)\big) u} du.\label{h,ab}
\end{equation}

Since $\Phi$ is bounded and  non-negative, the integral 
on the right-hand side of \eqref{h,ab} converges.
Since $h(u)$ is also non-negative,
the integral on the left-hand side converges absolutely. Moreover the integral 
\begin{equation*}
\quad \int_{0}^{\infty} \Phi(v - u) e^{-\varepsilon u} u^{\alpha_{i} - 1} du
\end{equation*}
 also converges absolutely.
Therefore, by Fubini's theorem, the following equation holds.
\begin{equation}
\begin{split}
&\int_{\mathbb{R}} \left( h(u) - C - \sum_{i=1}^{r} 
\frac{C_{i}}{u^{\alpha - \alpha_{i}}}\right) \Phi(v - u) 
e^{- \varepsilon u} u^{\alpha -1} du \\
 &=\int_{0}^{\infty} \left(h(u) - C - \sum_{i=1}^{r} 
\frac{C_{i}}{u^{\alpha - \alpha_{i}}}\right) 
e^{-\varepsilon u} u^{\alpha -1} 
\int_{\mathbb{R}} \mathcal{F}(\Phi)(t) e^{i(v - u)t} d\lambda(t) du \\
 &=  \int_{\mathbb{R}} \mathcal{F}(\Phi)(t) e^{ivt} \int_{0}^{\infty} 
\left( h(u) - C - \sum_{i=1}^{r}\frac{C_{i}}{u^{\alpha - \alpha_{i}}}
 \right) 
e^{-(it + \varepsilon) u} u^{\alpha -1} du d\lambda(t) \\ 
 &= \int_{\mathbb{R}} \mathcal{F}(\Phi)(t) e^{ivt} 
G(1 + \varepsilon + it) d\lambda(t) . \label{FH}
\end{split}
\end{equation}

By assumption, $G(s)$ is continuous on $\re (s) \ge 1$ 
and the support of $\mathcal{F}(\Phi)$ is compact.
Thus, $G(1 + \varepsilon + it)$ converges to  $G(1 + it)$ 
uniformly on the support of $\mathcal{F}(\Phi)(t)$ as $\varepsilon \to
 +0$. 
Therefore, 
\begin{equation}
\lim_{\varepsilon \to +0} \int_{\mathbb{R}} \mathcal{F}(\Phi)(t)
e^{ivt} G(1 + \varepsilon + it) d\lambda(t) 
= \int_{\mathbb{R}} \mathcal{F}(\Phi)(t) e^{ivt} G(1 + it) d\lambda(t). \label{F}
\end{equation}
As far as the left-hand side of \eqref{FH} is concerned, 
by using the monotone convergence theorem to each term,
we have 
\begin{equation}
\begin{split}
\lim_{\varepsilon \to +0} \int_{0}^{\infty} 
\Bigg( h(u) - &C - \sum_{i=1}^{r} \frac{C_{i}}{u^{\alpha - \alpha_{i}}}\Bigg)  
\Phi(v - u) e^{- \varepsilon u} u^{\alpha -1} du \\
&= \int_{\mathbb{R}} \left( h(u) - C - 
\sum_{i=1}^{r} \frac{C_{i}}{u^{\alpha - \alpha_{i}}} \right) 
\Phi(v - u) u^{\alpha -1} du . \label{H}              %2.5-----------------
\end{split}
\end{equation}

We may choose a closed-interval $[a,b]$ containing the support of $\mathcal{F}(\Phi)$.
By \cite[p.140 L{\footnotesize EMMA} 3]{kable}, if $\alpha < 1$
\begin{equation}
\begin{split}
\lim_{v \to \infty} 
& \ v^{1 - \alpha} \int_{\mathbb{R}} \mathcal{F}(\Phi)(t) 
e^{ivt} G(1 + it) d\lambda(t) \\
= & \lim_{v \to \infty} \ v^{1 - \alpha} 
\int_{a}^{b} \mathcal{F}(\Phi)(t) e^{ivt} G(1 + it) d\lambda(t) \\
= & \ 0 \label{vF}.
\end{split}
\end{equation}
If $\alpha \ge 1$, then we have the same 
equation by \cite[p.185, 7.5 Theorem]{fourier}.
Since the right-hand sides of \eqref{F} and 
\eqref{H} are equal, by using \eqref{vF}, we obtain
\begin{equation}
\lim_{v \to \infty} v^{1 - \alpha} \int_{0}^{\infty} \left(h(u) 
- C - \sum_{i=1}^{r} \frac{C_{i}}{u^{\alpha - \alpha_{i}}}\right) 
\Phi(v - u) u^{\alpha -1} du = 0 \label{A}%2,7--------------------
\end{equation}
for any $\alpha > 0$. 

By \cite[p.139, Lemma 1]{kable} we have
\begin{equation}
\lim_{v \to \infty} v^{1 - \alpha} \int_{0}^{\infty} 
\Phi(v - u) u^{\alpha - 1} du =  \int_{\mathbb{R}} \Phi(u) du \label{k.phi} 
\end{equation}
for any $\alpha > 0$. So, 
\begin{equation}
\begin{split}
\lim_{v \to \infty} v^{1 - \alpha}& \int_{0}^{\infty} 
\frac{C_{i}}{u^{\alpha - \alpha_{i}}} \Phi(v - u) u^{\alpha -1} du \\
&= \lim_{v \to \infty} \frac{1}{v^{\alpha - \alpha_{i}}} 
v^{1 - \alpha_{i}}\int_{0}^{\infty} C_{i}\Phi(v - u) u^{\alpha_{i} -1} du  \\
&= \ 0 	\label{C}		
\end{split}
\end{equation}
for $i = 1, 2, \cdots, r$. Therefore, by \eqref{A}, \eqref{k.phi}, and \eqref{C}, we have
\begin{equation}
\begin{split}
\lim_{v \to \infty} v^{1 - \alpha}\!\!\int_{0}^{\infty} h(u) \Phi(v - u) u ^{\alpha -1} du
&=\lim_{v \to \infty} C \ v^{1 - \alpha} \!\!\int_{0}^{\infty}\! \Phi(v - u) u^{\alpha - 1} du \notag \\ 
&= C \int_{\mathbb{R}} \Phi(u) du .
\end{split}
\end{equation}

The rest of the argument is similar to that 
in \cite[pp.142-143]{kable}, and we can conclude that 
\begin{equation*}
\lim_{u \to \infty} h(u) = C.
\end{equation*}
Therefore, Theorem \ref{thm:preliminary-theorem} is proved.
\end{proof}

\section{Main Theorem}
\label{sec:main-theorem}

In this section, we apply Theorem \ref{thm:preliminary-theorem} to Dirichlet series 
which are obtained by sequences of non-negative real numbers and satisfy some conditions.
The following theorem is the main application of Theorem \ref{thm:preliminary-theorem}.
\begin{thm}
\label{thm:main-theorem}
Let $\{a_{n}\}$ be a sequence of non-negative real numbers, 
$d \in \mathbb{R}_{> 0}$ and $m$ a positive integer. 
Suppose that the Dirichlet series
\begin{equation*}
L(s) = \sum_{n=1}^{\infty} \frac{a_{n}}{n^s}   
\end{equation*}
converges absolutely for $\re(s) > d$. 
Also suppose that $L^m$ has a meromorphic continuation 
to an open set containing the closed half-plane 
$\re(s) \ge d$ and holomorphic except for a pole of order $l$ at $s = d$ with
$\lim_{s \to d} L(s)^m (s - d)^l = A^m$, where $A > 0$. Then we have 
\begin{equation*}
\sum_{n \le X} a_{n} \sim \frac{AX^d}{d \varGamma (l / m) (\log(X))^{1 - \frac{l}{m}}} 
\end{equation*}
as $X \to \infty$.
\end{thm}
\begin{proof}
For a subset $S \subset \R $, let $\phi_{S}$ be the characteristic function of $S$.
Define a function $f : \mathbb{R} \to [0 ,\infty)$ by 
\begin{equation*}
f(u) = \sum_{n=1}^{\infty} a_{n} \phi_{[d \log(n),\infty)}(u). 
\end{equation*}
By direct computation, we have 
\begin{equation*}
F(s) = \int_{0}^{\infty} f(u) e^{-su} du = \frac{1}{s} L(ds) 
\end{equation*}
for $\re(s) > 1$.

Let $L(s)^m (s - d)^l = Q(s)$.
Since $L(s)^m$ has a pole of order $l$ at $s = d$, 
$Q(s)$ is holomorphic around $s = d$ and $Q(d) = A^{m} ( \neq 0)$. 
So there exists a holomorphic function $P(s)$ 
defined on an open disc $D$  with center at $s = d$ 
such that $P(s)^{m} = Q(s)$. Since $L(s)(s - d)^{\frac{l}{m}}$ 
is holomorphic on $\re(s) > d$,
there exists an $m$-th root of unity $\zeta$ such that 
$L(s)(s - d)^{\frac{l}{m}} = \zeta\cdot P(s)$ on 
$\{s \in \C \ | \ \re(s) > d \} \cap D$. 
Therefore, $L(s)(s - d)^{\frac{l}{m}}$ can be 
extended to a holomorphic function around $s = d$, 
and $\lim_{s \to d} L(s)(s -d)^{\frac{l}{m}} = A \ (A \in
\mathbb{R}_{>0})$ 
since $L(s)(s - d)^{\frac{l}{m}}$ 
has positive real values for $s \in \mathbb{R}_{> d}$.  
Hence, $\frac{1}{s}L(ds) (d s - d)^{\frac{l}{m}}$ is holomorphic around $s = 1$.
Since 
\begin{equation*}
F(s)(s - 1)^{\frac{l}{m}} = \frac{1}{s} L(ds) (s - 1)^{\frac{l}{m}} 
= \frac{1}{sd^{\frac{l}{m}}} L(ds) (ds - d)^{\frac{l}{m}},
\end{equation*}
$F(s)(s - 1)^{\frac{l}{m}}$ is holomorphic around $s = 1$.

Therefore, there exists a holomorphic function $B(s)$ 
defined on an open disc $D'$ with center at $s = 1$ such that
\begin{equation*}
F(s)(s - 1)^{\frac{l}{m}} = \sum_{i=0}^{r} A_{i}(s -1)^i + (s -1)^{r + 1} B(s),
\end{equation*}
where $r = \ceil{ l /m } -1$ and $A_{i} \in \C$ for $i = 0, 1, \cdots, r$.
Then
\begin{equation*}
\lim_{s \to 1} F(s)(s - 1)^{\frac{l}{m}} = \frac{A}{d^{\frac{l}{m}}} 
\ \ \ i.e., \ \ A_{0} = \frac{A}{d^{\frac{l}{m}}}.
\end{equation*}
Thus, we have 
\begin{equation*}
F(s) = \frac{A_{0}}{(s - 1)^{ \frac{l}{m}}} + \sum_{i=1}^{r} 
\frac{A_{i}}{(s - 1)^{ \frac{l}{m} - i}} + (s -1)^{r + 1 - \frac{l}{m}} B(s)
\end{equation*}
for $s \in D' \cap \{ s \in \mathbb{C} \ | \ \re(s) > 1 \}$. 

Now, we define
\begin{equation*}
\begin{split}
G(s) &= F(s) - \frac{A_{0}}{(s - 1)^{\frac{l}{m}}} 
- \sum_{i=1}^{r} \frac{A_{i}}{(s - 1)^{\frac{l}{m} - i}}\\
&= \frac{1}{s} L(ds) - \frac{A_{0}}{(s - 1)^{\frac{l}{m}}} 
- \sum_{i=1}^{r} \frac{A_{i}}{(s - 1)^{\frac{l}{m} - i}}
\end{split}
\end{equation*}
for $\re(s) > 1$.
  
We shall prove the following (a), (b).
\begin{flushleft}
\begin{itemize}
\item[(a)] 
The function $G(s)$ extends continuously to 
the set $\{ s \in \mathbb{C} \ | \ \re(s) \ge 1 \}$.
\item[(b)]
The function $x \mapsto G(1 + ix)$ $(x \in \R)$ is locally of bounded variation.
\end{itemize}
\end{flushleft}
Then $f(s)$, $F(s)$, and $G(s)$ satisfy the 
condition of Theorem \ref{thm:preliminary-theorem}.
It is enough to prove (a), (b) locally.

We first consider the neighborhood of $s = 1$.
By definition 
\begin{equation*}
G(s) = (s - 1)^{r + 1 - \frac{l}{m}} B(s) 
\end{equation*}
for $s \in D' \cap \{ s \in \mathbb{C} \ | \ \re(s) > d \}$.

By Theorems 10.2 of \cite[p.192]{real}, 
the function $x \mapsto B(1 + ix)$ is 
of bounded variation on any closed-interval $[a,b]$, 
where $\{1 + ix \in \mathbb{C} \ | \ x \in [a,b] \}  \subset D'$.
For any $\alpha > 0$ we can obviously 
extend $(s - 1)^{\alpha}$ to a continuous function on 
$\re(s) \ge 1$.  By a similar argument as in \cite[p.144]{kable}, 
the function $x \mapsto \bigl((1 + ix) - 1\bigr)^{\alpha}$ $(x \in \R)$ 
is locally of bounded variation.   

Therefore, if $r + 1 - \frac{l}{m} > 0$, then $G(s)$ extends 
continuously to the set $D' \cap \{ s \in \C \ | \ \re(s) \ge 1 \}$ 
and the function $x \mapsto G(1 + ix)$ ($x \in \mathbb{R}$) is of 
bounded variation on any closed-interval $[a,b]$, where 
$\{1 + ix \in \mathbb{C} \ | \ x \in [a,b] \}  \subset D'$. 
If $r + 1 - \frac{l}{m} = 0$, then $G(s)$ has the same properties 
since $G(s) = B(s)$ on $D' \cap \{ s \in \C \ | \ \re(s) > 1 \} $.
Thus, (a), (b) are proved around $s = 1$.

Next we consider the neighborhood of $s = 1 + ix_{0}$, 
where $x_{0} \in \R - \{0\}$. By the definition of $G(s)$, 
it is enough to prove (a), (b) for $L(ds)$ instead of $G(s)$.

Since $L^{m}$ is holomorphic around $s = d + ix_{0}$, 
there exists a holomorphic function $R(s)$ 
defined for $s$ near $d + ix_{0}$ such that 
\begin{equation*}
L(s)^m = \bigl(s - (d + ix_{0}) \bigr)^k R(s),  \ \ 
R(d + ix_{0}) \neq 0, \  \ and \ \ \  k \ge 0.
\end{equation*}
Since $R(d + ix_{0}) \neq 0$, there exists a 
holomorphic function $T(s)$ defined on an open disc 
$D''$ with center at $s = d + ix_{0}$ such that $T(s)^{m} = R(s)$.
Then, there exists an $m$-th root of unity  $\zeta$ such that  
\begin{equation*}
L(s) = \zeta \cdot \bigl(s - (d + ix_{0}) \bigr)^{\frac{k}{m}} T(s) 
\end{equation*}
on $D'' \cap \{ s \in \mathbb{C} \ | \ \re(s) > d\}$.

Now, for any $x_{0} \in \R - \{0\}$ we  can extend 
$\bigl(s - (d + ix_{0}) \bigr)^{\frac{k}{m}}$ 
to a continuous function on $\re(s) \ge d$ and 
the function $x \mapsto \bigl(d + ix - (d + ix_{0})\bigr)^{\frac{k}{m}}$ 
is locally of bounded variation also as in \cite[p.144]{kable}.
Thus, $L(s)$ extends continuously to the set  
$D'' \cap \{ s \in \mathbb{C} \ | \ \re(s) \ge d\}$ 
and the function $x \mapsto L(d + ix)$ is of 
bounded variation around $x_{0}$.  It follows 
that we can extend  $L(ds)$ to a continuous function 
on $\re(s) \ge 1$ except $s = 1$
and the function $x \mapsto L(d + idx)$ is of 
bounded variation on any closed-interval $[a,b]$ 
not containing $x = 0$. Thus, we have proved (a), (b).

Therefore, we can apply Theorem \ref{thm:preliminary-theorem} 
to $f(s)$, $F(s)$, $G(s)$, and  $\alpha = l / m$. So we have
\begin{equation}
\lim_{u \to \infty} u^{1 - \frac{l}{m}} e^{-u} f(u) 
= \frac{A_{0}}{\varGamma(l / m)} 
= \frac{A}{d^{\frac{l}{m}}\varGamma(l / m)}. \label{last}
\end{equation}
Since $f(d \log X ) = \sum_{n \le X} a_{n}$, by substituting  
$d\log X$ for $u$ in \eqref{last}, we have
\begin{eqnarray*}
\lim_{X \to \infty} (d \log X)^{1 - \frac{l}{m}} e^{-(d \log X)} f(d\log X) 
= \frac{A}{d^{\frac{l}{m}}\varGamma(l / m)}, \\
\lim_{X \to \infty} d^{1 - \frac{l}{m}} (\log X)^{1 - \frac{l}{m}} X^{-d} 
\sum_{n \le X} a_{n} = \frac{A}{d^{\frac{l}{m}}\varGamma(l / m)}.
\end{eqnarray*}
Therefore,
\begin{equation*}
\sum_{n \le X} a_{n} \sim \frac{AX^d}{d \varGamma (l / m) (\log(X))^{1 - \frac{l}{m}}}. 
\end{equation*}
\end{proof}

\section{An Example}
\label{sec:an-example}

We give an example of a Dirichlet series having a pole of 
order "$2 / 3$" and apply Theorem \ref{thm:main-theorem}.
 
If $n = {p_{1}}^{e_{1}}{p_{2}}^{e_{2}}\cdots{p_{r}}^{e_{r}}$ 
is a prime decomposition of a positive integer $n$ 
where $p_1,\ldots,p_r$ are distinct primes, then 
we define $f(n) = \sum_{n = 1}^{r} e_{i}$.
We define a Dirichlet series $L$ by
\begin{equation*}
L(s) = \prod_{p : prime} \left(1 - \frac{2}{3} p^{-s} \right)^{-1} 
= \sum_{n = 1}^{\infty} \left(\frac{2}{3} \right)^{f(n)} n^{-s}. 
\end{equation*}
Then the function $L(s)$ is holomorphic on $\re(s) > 1$ 
since $\left(2/3 \right)^{f(n)} \le 1$. 
We claim that $L(s)^3$ has a meromorphic continuation 
to an open set containing $\re(s) \ge 1$ and 
holomorphic except for a pole of order $2$ at $s = 1$.

By computation
\begin{equation*}
L(s)^{3} = \prod_{p : prime} \left(1 - 2p^{-s} 
+ \frac{4}{3}p^{-2s} - \frac{8}{27}p^{-3s} \right)^{-1},
\end{equation*}
and
\begin{equation*}
\begin{split}
 \frac{L(s)^{3}}{\zeta(s)^2} &= \prod_{p : prime} \frac{(1 - p^{-s})^{2}}
{ \left(1 - 2p^{-s} + \frac{4}{3}p^{-2s} - \frac{8}{27} p^{-3s}\right)}\\
&= \prod_{p : prime} \frac{(1 - p^{-s})^2(1 + p^{-s})^{2}}{ \left(1 - 
2p^{-s} + \frac{4}{3}p^{-2s} - \frac{8}{27}p^{-3s}\right)(1 + p^{-s})^2}\\
&= \prod_{p : prime} \frac{(1 - p^{-2s})^{2}}{ \left(1 - \frac{5}{3}p^{-2s} 
+ \frac{10}{27}p^{-3s} + \frac{20}{27}p^{-4s} - \frac{8}{27}p^{-5s} \right) }\\
&= \zeta(2s)^{-2} \prod_{p : prime}\left(1 - \frac{5}{3}p^{-2s} 
+ \frac{10}{27}p^{-3s} + \frac{20}{27}p^{-4s} - \frac{8}{27}p^{-5s} \right)^{-1}.
\end{split}
\end{equation*}
Let 
\begin{equation*}
F(s) = \prod_{p : prime} \left(1 - \frac{5}{3}p^{-2s} + \frac{10}{27}p^{-3s} 
+ \frac{20}{27}p^{-4s} - \frac{8}{27}p^{-5s} \right).
\end{equation*}

If $\re(s) \ge 0$, then we have
\begin{equation*}
\left|- \frac{5}{3}p^{-2s} + \frac{10}{27}p^{-3s} + \frac{20}{27}p^{-4s} 
- \frac{8}{27}p^{-5s} \right| \le \frac{83}{27}|p^{-2s}|.
\end{equation*}
Thus, if $\re(s) > 1/2$, then
\begin{equation*}
\begin{split}
\sum_{p : prime } \left| -\frac{5}{3}p^{-2s} + \frac{10}{27}p^{-3s} 
+ \frac{20}{27}p^{-4s} - \frac{8}{27}p^{-5s}\right|  
&\le \sum_{p : prime } \frac{83}{27}p^{-2\re(s)}\\
&< +\infty.
\end{split}
\end{equation*}
Therefore, 
\begin{equation*}
\sum_{p : prime } \left|- \frac{5}{3}p^{-2s} + \frac{10}{27}p^{-3s} 
+ \frac{20}{27}p^{-4s} - \frac{8}{27}p^{-5s}\right|
\end{equation*}
converges absolutely and uniformly on any compact 
subset of the open half-plane $\re(s) > 1 / 2$.
It follows that 
\begin{equation*}
\prod_{p : prime }  \left(1 - \frac{5}{3}p^{-2s} + \frac{10}{27}p^{-3s} 
+ \frac{20}{27}p^{-4s} - \frac{8}{27}p^{-5s}\right)
\end{equation*}
converges uniformly on any compact subset of the open 
half-plane $\re(s) > 1 / 2$ and is holomorphic on 
$\re(s) > 1/2$ (see \cite[p.300, 15.6 Theorem ]{complex}).
Thus, $F(s)$ is holomorphic on $\re(s) > 1/2$.

For any prime number $p$, since $(1 -\frac{2}{3}p^{-s})$, 
$(1 + p^{-s}) \neq 0$ on $\re(s) > 1/2$,
\begin{equation*}
\left(1 - \frac{5}{3}p^{-2s} + \frac{10}{27}p^{-3s} 
+ \frac{20}{27}p^{-4s} - \frac{8}{27}p^{-5s} \right) 
= \left(1 + p^{-s}\right)^{2}\left(1 -\frac{2}{3}p^{-s}\right)^{3} \neq 0
\end{equation*}
on $\re(s) > 1/2$. 
Hence $F(s) \neq 0$ on $\re(s) > 1/2$ i.e., $F(s)^{-1}$ is  
holomorphic on $\re(s) > 1/2$ (see \cite[p.300, 15.6 Theorem]{complex}).

It follows that $L(s)^3$ has a meromorphic continuation to 
$\re(s) > 1 / 2$ and is holomorphic except 
for a pole of order $2$ at $s = 1$.  Moreover, 
\begin{equation*}
\begin{split}
\lim_{s \to 1} L(s)^{3}(s - 1)^2 &= \lim_{s \to 1} 
\frac{L(s)^{3}}{\zeta(s)^2} \left(\zeta(s)(s - 1) \right)^2 \\
&=  \prod_{p : prime} \frac{(1 - p^{-2})^{2}}{\left(1 - \frac{5}{3}p^{-2} 
+ \frac{10}{27}p^{-3} + \frac{20}{27}p^{-4} - \frac{8}{27}p^{-5}\right) }\\
&= \zeta(2)^{-2} F(1)^{-1}.
\end{split}
\end{equation*}

By applying Theorem \ref{thm:main-theorem} to $L(s)$, we have
\begin{equation*}
\sum_{n \le X} \left(\frac{2}{3} \right)^{f(n)} 
\! \sim \ \ \frac{A}{\Gamma(2 / 3)} \cdot \frac{X}{\log(X)^{1/3}},  
\end{equation*}
where 
\begin{equation*}
A = \left( \zeta(2)^{2} \prod_{p : prime} 
\left(1 - \frac{5}{3}p^{-2} + \frac{10}{27}p^{-3} 
+ \frac{20}{27}p^{-4} - \frac{8}{27}p^{-5}\right) \right)^{-\frac{1}{3}}. 
\end{equation*}

\bibliographystyle{plain}
\bibliography{katorefs}

\end{document}